\documentclass[12pt,reqno]{amsart}
\usepackage[margin=1in]{geometry}
\usepackage{amscd,amsfonts,amsmath,amssymb,amsthm,latexsym,mathrsfs,textcomp,verbatim}
\usepackage{accents}
\usepackage{bm}
\usepackage{bbm}
\usepackage{booktabs}
\usepackage{cite}
\usepackage{color}
\usepackage{constants}
\usepackage{csquotes}
\usepackage{enumerate}
\usepackage{etoolbox}
\usepackage{float}
\usepackage{hypbmsec}
\usepackage[dvipsnames]{xcolor}
\usepackage[breaklinks=true,colorlinks=true,linkcolor=black!25!RoyalBlue,citecolor=black!25!RoyalBlue,filecolor=RoyalBlue,urlcolor=BlueViolet]{hyperref}
\hypersetup{linktocpage}
\usepackage[capitalise]{cleveref}
\usepackage{mathtools}
\usepackage{soul}
\usepackage{tikz}
\usepackage{tikz-cd}
\usetikzlibrary{shapes.geometric}
\usetikzlibrary{shapes.misc}
\usetikzlibrary{positioning}
\allowdisplaybreaks[4]
\usepackage{hypbmsec}
\newtheorem{theorem}{Theorem}[section]
\newtheorem{corollary}[theorem]{Corollary}
\newtheorem{proposition}[theorem]{Proposition}

\newtheorem{conjecture}[theorem]{Conjecture}

\newtheorem{remark}[theorem]{Remark}
\newtheorem{definition}[theorem]{Definition}
\newtheorem{question*}{Question}

\theoremstyle{definition}

\newtheorem{example}[theorem]{Example}

\numberwithin{equation}{section}

\newcommand{\Z}{\mathbb{Z}}

\newcommand{\R}{\mathbb{R}}

\newcommand{\GL}{\mathrm{GL}}

\newcommand{\SL}{\mathrm{SL}}

\renewcommand{\pmod}[1]{\, (\mathrm{mod} {\, #1})}

\newcommand{\ord}{\mathop{\mathrm{ord}}}

\renewcommand{\Re}{\mathrm{Re}}

\patchcmd{\section}{\scshape}{\bfseries}{}{}
\makeatletter
\renewcommand{\@secnumfont}{\bfseries}
\makeatother

\makeatletter\newcommand{\tpmod}[1]{{\@displayfalse\pmod{#1}}}
\makeatletter
\@namedef{subjclassname@2020}{\textup{2020} Mathematics Subject Classification}
\makeatother

\begin{document}

\title{Chebyshev's bias for modular forms}

\author{Shin-ya Koyama}
\address[Shin-ya Koyama]{Department of Mechanical Engineering, Toyo University, 2100 Kujirai, Kawagoe, Saitama, 350-8585, Japan}
\email{koyama@tmtv.ne.jp}

\author{Arshay Sheth}
\address[Arshay Sheth]{Mathematics Institute, Zeeman Building, University of Warwick, Coventry, CV4 7AL, United Kingdom }
\curraddr{School of Mathematics, Tata Institute of Fundamental Research, Homi Bhabha Road,
Mumbai - 400005, India}
\email{arshay.sheth@warwick.ac.uk, asheth@math.tifr.res.in}

\thanks{The first author is partially supported by Enryo INOUE Memorial grant 2024, Toyo University.}
\thanks{The second author was supported by funding from the European Research Council under the European Union’s Horizon 2020 research and innovation programme (Grant agreement No. 101001051 — Shimura varieties and the Birch--Swinnerton-Dyer conjecture) during the writing of this paper.}


\begin{abstract}
We study Chebyshev's bias for the signs of Fourier coefficients of cuspidal newforms on $\Gamma_0(N)$. Our main result shows that the bias towards either sign is completely determined by the order of vanishing of the $L$-function $L(s, f)$ at the central point of the critical strip.  We then give several examples of modular forms where we explicitly compute the order of vanishing of $L(s, f)$ at the central point and as a by-product, verify the super-positivity property, in the sense of Yun--Zhang (2017), for these examples. 
\end{abstract}

\maketitle

\section{Introduction}

Chebyshev's bias originally referred to the phenomenon that, even though the primes are equidistributed in the multiplicative residue classes mod 4, there seem to be more primes congruent to 3 mod 4 than 1 mod 4. This was first noted by Chebyshev in a letter to Fuss in 1853, and supported by extensive numerical evidence over the next few decades. However, a rigorous mathematical formulation of the bias remained out of reach for the next century, with many natural formulations running into stumbling blocks. For example, let $\pi(x; q, a)$ denote the number of primes up to $x$ congruent to $a$ modulo $q$ and let $S=\{x \in \mathbb R_{\geq 2}: \pi(x; 4, 3)-\pi(x; 4, 1)>0\}$; then Knapowski--Tur\'{a}n~\cite{KnapowskiTuran1962} conjectured that the proportion of positive real numbers lying in the set $S$ would equal 1 as $x \to \infty$.  However, this conjecture was later disproven by Kaczorowski~\cite{Kaczorowski1995}  conditionally on the Generalised Riemann Hypothesis, by showing that the limit does not exist. 

A breakthrough in this direction of formulating the bias came via the work of Rubinstein and Sarnak \cite{RubinsteinSarnak1994}, who instead considered the logarithmic density $$\delta(S):= \lim  \limits _{X \to \infty}  \frac{1}{\log X} \cdot \int_{t \in S \cap [2, X]} \frac{dt}{t}$$ of $S$; assuming the Generalised Riemann Hypothesis and that the non-negative imaginary parts of zeros of Dirichlet $L$-functions are linearly independent over $\mathbb Q$, they showed that this limit exists and $\delta(S)=0.9959 \ldots$, hence giving a satisfactory explanation of Chebyshev's observations. 
The phenomenon of Chebyshev's bias has since been observed in various other situations; see, for instance, the expository article \cite{Maz08} by Mazur  discussing Chebyshev's bias in the context of elliptic curves, as well as 
the annotated bibliography \cite{Martin} by Martin \textit{et al.} which contains a detailed overview of related literature on the subject.


In \cite{AokiKoyama2023}, Aoki--Koyama introduced a new framework of studying Chebyshev's bias which crucially relies on the Deep Riemann Hypothesis, a conjecture by Kurokawa \textit{et al.} about the convergence of Euler products of $L$-functions on the critical line.
We now briefly explain this conjecture in the setting of automorphic $L$-functions. Let $\pi$ be an irreducible cuspidal automorphic representation of $\mathrm{GL}_n$ over $\mathbb{Q}$ with associated $L$-function $L(s, \pi)$. We can write $L(s, \pi)$ as 
$$
L(s, \pi)=\prod_p \prod_{j=1}^n (1-\alpha_{j, p}p^{-s})^{-1},  
$$
where, for the unramified primes $p$, the $\alpha_{j, p}$'s are the Satake parameters for the local representation $\pi_p$ (see \S \ref{2}).  We let  $\nu(\pi) = m(\mathrm{sym}^{2} \pi)-m(\wedge^{2} \pi) \in \Z$,  where $m(\rho)$ denotes the multiplicity of the trivial representation $\textbf{1}$ in $\rho$.  
We have $\nu(\pi)=-\ord_{s=1} L_2(s, \pi)$, where $L_2(s, \pi)$ is the second moment $L$-function attached to $L(s, \pi)$ defined via 
\begin{equation} \label{sm}
L_2(s, \pi)=\prod_p  \prod_{j=1}^n (1-\alpha_{j, p}^2 p^{-s})^{-1}. 
\end{equation}

\begin{conjecture}[Kaneko--Koyama--Kurokawa \cite{KanekoKoyamaKurokawa2022}]\label{DRH}
Keep the assumptions and notation as above and assume that $L(s, \pi)$ is entire. Let $m = \ord_{s = 1/2} L(s, \pi)$. Then the limit
\begin{equation}\label{limit}
\lim_{x \to \infty} \left((\log x)^{m} \prod_{p \leq x} \prod_{j=1}^n \left(1-\alpha_{j, p} p^{-\frac{1}{2}} \right)^{-1} \right)
\end{equation}
satisfies the following conditions:
\begin{description}
\item[(A)] The limit~\eqref{limit} exists and is non-zero.
\item[(B)] The limit~\eqref{limit} satisfies
\begin{equation*}
\lim_{x \to \infty} \left((\log x)^{m} \prod_{p \leq x} \prod_{j=1}^n \left(1-\alpha_{j, p} p^{-\frac{1}{2}} \right)^{-1} \right)
 = \frac{\sqrt{2}^{ \nu (\pi) }}{e^{m \gamma} m!} \cdot L^{(m)} \left(\frac{1}{2}, \pi \right).
\end{equation*}
\end{description}
\end{conjecture}

The constant $\frac{\sqrt{2}^{ \nu (\pi) }}{e^{m \gamma}}$ in the Euler product asymptotic in Conjecture \ref{DRH} (B) was first discovered by Goldfeld \cite{Goldfeld1982} in the setting of $L$-functions of elliptic curves, while studying the original version of the Birch and Swinnerton--Dyer conjecture. Euler products on the critical line for general $L$-functions were subsequently studied in great depth by Conrad \cite{Conrad2005}, whose results motivated the formulation of Conjecture \ref{DRH}. 

Conjecture \ref{DRH} implies the Generalised Riemann Hypothesis (GRH) for $L(s, \pi)$, and is deeper than GRH in a precise quantitative sense; we refer to \cite{KanekoKoyamaKurokawa2022} for more background on the conjecture. There are both theoretical considerations as well as extensive numerical evidence which support Conjecture \ref{DRH}. For instance, the function-field analogue of the conjecture is known (\cite[Theorem 5.2]{KanekoKoyamaKurokawa2022}) and the second author showed in \cite{Sheth24B} that the GRH implies Conjecture \ref{DRH} as $x \to \infty$ outside a set of finite logarithmic measure.

We now recall the framework of Aoki--Koyama \cite{AokiKoyama2023} for studying Chebyshev's bias using Conjecture \ref{DRH}. 

\begin{definition}[Aoki--Koyama \cite{AokiKoyama2023}] \label{Chebyshev}
Let $(c_p)_p \subseteq \R$ be a sequence over primes $p$  such that
$$
\lim_{x \to \infty} \frac{\#\{p \mid c_p > 0,  p \leq x \}}{\#\{p \mid c_p < 0,  p \leq x \}} = 1.
$$
We say that $(c_p)_p$ has a \textit{Chebyshev's bias towards being positive} (resp. negative) if there exists a positive (resp. negative) constant $C$ such that 
\begin{equation*}
\sum_{p \leqslant x} \frac{c_p}{\sqrt p} \sim C \log \log x.\end{equation*}  On the other hand, we say that $c_p$ is \textit{unbiased} if
\begin{equation*}
\displaystyle{\sum_{p \leq x} \frac{c_p}{\sqrt p } = O(1)}. 
\end{equation*} 
\end{definition}

\begin{example}[{Aoki--Koyama \cite[Example 3.4]{AokiKoyama2023}}]
Let $\chi_4$ denote the non-trivial Dirichlet character mod 4, so $\chi_4(p)=-1$ if $p \equiv 3 \textrm{ mod } 4$ and $\chi_4(p)=1$ if $p \equiv 1 \textrm{ mod } 4$. Assume Conjecture \ref{DRH} for $L(s, \chi_4)$.
Then there is a constant $c$ such that
$$
\sum_{p \leqslant x} \frac{\chi_4(p)}{\sqrt p}= -\frac{1}{2} \log \log x+ c+ o(1). 
$$
Thus, in the sense of Definition \ref{Chebyshev}, there is a Chebyshev's bias towards primes which are 3 mod 4. 
\end{example}

\begin{example}[{Koyama--Kurokawa \cite[Theorem~2]{KK22}}]\label{tau}
Let $\tau(n)$ denote Ramanujan's tau function. Assume Conjecture~\ref{DRH} for $L(s+\frac{11}{2}, \Delta)$. Then there exists a constant $c$ such that
\begin{equation*}
\sum_{p \leqslant x} \frac{\tau(p)}{p^{6}} = \frac{1}{2} \log \log x+c+o(1).
\end{equation*}
Thus, in the sense of Definition \ref{Chebyshev}, the sequence $(\tau(p) p^{-\frac{11}{2}})_p$ has a Chebyshev's bias towards being positive. 
\end{example}

We refer to \cite{AokiKoyama2023, KanekoKoyama2023, Okumura2024} for further examples of Chebsyhev's bias in the framework of Definition \ref{Chebyshev}. The main theorem of this paper is to generalise Example \ref{tau} by establishing the desired asymptotic in Definition \ref{Chebyshev} when $c_p=a_f(p)$ is the Fourier coefficient of an arbitrary cuspidal newform on $\Gamma_0(N)$. As we explain in \S \ref{statement} below, the signs of the Fourier coefficients $a_f(p)$ are equidistributed as we range over all primes $p$ (asymptotically half of them are positive and half of them are negative), so we are indeed in the setting of Definition \ref{Chebyshev} and it is thus natural to investigate a bias towards either sign. 

\subsection{Statement of the main result.} \label{statement}

Let $f(z)=\sum_{n=1}^{\infty}a_f(n)n^{\frac{k-1}{2}}e^{2\pi inz}\in S_k(\Gamma_0(N))$ be a  cusp form (normalised so that $a_f(1)=1$) with trivial nebentypus. The restriction to trivial nebentypus implies that $a_f(n)$ is real for all $n \geq 1$. We recall that if $f$ is also an eigenform for all of the Hecke operators and Atkin--Lehner involutions, then $f$ is called a newform.   In this paper, we assume that $f$ is not a CM form, so there is no imaginary quadratic field $K$ such that for $p\nmid N$, $p$ is inert in $K$ if and only if $a_f(p)=0$.  Deligne's proof of the Weil conjectures 
implies that for each prime $p$, there exists an angle $\theta_p\in[0,\pi]$ such that $a_f(p) = 2\cos\theta_p$. The Sato--Tate conjecture, now proven by Barnet-Lamb, Geraghty, Harris and Taylor \cite{BLGHT11}, 
asserts that if $f$ is non-CM, then the sequence $\{\theta_p\}$ is equidistributed in the interval $[0,\pi]$ with respect to the measure $d\mu_{\mathrm{ST}}:=(2/\pi)\sin^2\theta d\theta$.  Equivalently,  for an interval $I \subseteq [0, \pi]$, one has $\{p\leq x\colon \theta_p\in I\} \sim \mu_{\mathrm{ST}}(I)\pi(x)$ as $x\to\infty$. In particular, the Fourier coefficients $a_f(p)$ satisfy the condition in Definition \ref{Chebyshev} (the set of primes $p$ for which $a_f(p)=0$ has density zero by \cite[Corollaire 2, p.174]{Ser81}). Let $\alpha_p = e^{i \theta_p}$ and $\beta_p = e^{-i \theta_p}$ so that
\begin{equation} \label{ap}
a_f(p) = \alpha_p+\beta_p \text{ and } \alpha_p \beta_p =1. 
\end{equation}

The $L$-function of $f$ is defined to be 
$$
L(s, f):= \sum_{n=1}^{\infty} \frac{a_f(n)}{n^s}= \prod_{p} (1- a_f(p) p^{-s}+p^{-2s})^{-1}  = \prod_{p} (1-\alpha_p p^{-s})^{-1} (1-\beta_p p^{-s})^{-1}, 
$$
which satisfies a functional equation relating $s$ to $1-s$ and admits an analytic continuation to the entire complex plane. 

\begin{theorem} \label{mainthm6}
Let $f \in S_k(\Gamma_0(N))$ be a cuspidal newform. 
Assume Conjecture \ref{DRH} for $L(s, f)$ and let $m(f)=\ord_{s=1/2} L(s, f)$.  Then there exists a constant $c_f$ such that 
$$
\sum_{p \leqslant x} \frac{a_f(p) }{\sqrt p} = \left( \frac{1}{2}- m(f) \right) \log \log x +c_f +o(1). 
$$
%
In particular, in the sense of Definition \ref{Chebyshev}, we conclude that 
\begin{enumerate}
 
\item The sequence $(a_f(p))_p$ has a Chebyshev's bias towards being positive if $m(f)=0$. 

\item The sequence $(a_f(p))_p$ has a Chebyshev's bias towards being negative if $m(f) \geq 1$.  
\end{enumerate}
\end{theorem}

To the best of our knowledge, Theorem \ref{mainthm6} is the first result in the literature which deals with the question of Chebyshev's bias for the signs of Fourier coefficients of arbitrary cuspidal newforms on $\Gamma_0(N)$. Another feature of Theorem \ref{mainthm6}, which is common with the other results in the Aoki--Koyama framework, is that it does not rely on the Linear Independence  Hypothesis;  
we remark that using \cite[Theorem 4.1]{Sheth24B}, the asymptotic in Theorem \ref{mainthm6} holds as $x \to \infty$ outside a set of finite logarithmic measure assuming only GRH, and so we can in fact obtain a statement which captures the bias assuming only GRH.  
Since Theorem \ref{mainthm6} shows that the bias is completely determined by $m(f)$, in \S \ref{4} we explicitly compute $m(f)$ for various examples of Hecke cusp forms for the group $\SL_2(\mathbb Z)$, thereby obtaining completely explicit versions of Theorem \ref{mainthm6} for these examples. 
\subsection*{Acknowledgements}
The authors are grateful to Nobushige Kurokawa and Hidekazu Tanaka for sharing their unpublished  manuscript on Mizumoto's positivity, on which \S \ref{miz} is based. We also thank Adam Harper and Chung-Hang Kwan for helpful remarks on a previous draft of this manuscript, and Brian Conrey and Jeremy Rouse for helpful correspondence. 
\section{Proof of Theorem \ref{mainthm6} } \label{2}

We prove Theorem \ref{mainthm6} by applying a general asymptotic about Chebyshev's bias for Satake parameters of cuspidal automorphic representations, proven in \cite[Theorem 4.1]{Sheth24B}. We begin by introducing the relevant definitions and notations. 

\label{Section2}
Let $\pi=\bigotimes'\pi_v$ be an irreducible cuspidal automorphic representation of $\GL_n(\mathbb A_\mathbb Q)$.  Outside a finite set of places, for each finite place $p$, $\pi_p$ is unramified and we can associate to $\pi_p$ a semisimple conjugacy class $\{A_\pi(p)\}$ in $\GL_n(\mathbb C)$. Such a conjugacy class is parametrised by its eigenvalues $\alpha_{1, p}, \ldots, \alpha_{n, p}$. The local Euler factors $L_p(s, \pi_p)$ are given by 
$$
L_p(s, \pi_p)=\det(1-A_\pi(p) p^{-s})^{-1}= \prod_{j=1}^n (1-\alpha_{j, p}p^{-s})^{-1}. 
$$
At the ramified finite primes, the local factors are best described by the Langlands parameters of $\pi_p$ (see for instance the Appendix in \cite{SarnakRudnick1996}). They are of the form $L_p(s, \pi_p) = P_p(p^{-s})^{-1}$, where $P_p(x)$ is a polynomial of degree at most $n$, and $P_p(0)$ = 1. We will in this case too write the local factors in the form above, with the convention that we now allow some of the $\alpha_{j, p}$'s to be zero. The global $L$-function attached to $\pi$ is given by
\begin{align*}
L(s, \pi)= \prod_p L_p(s, \pi_p) &=\prod_p  \prod_{j=1}^n (1-\alpha_{j, p} p^{-s})^{-1}. 
\end{align*}

Up to finitely many local Euler factors, the Rankin--Selberg, symmetric square and exterior square $L$-functions are given by 
\begin{equation*}
L(\pi \otimes \pi , s) \overset{.}= \prod_p \prod_{i=1}^n \prod_{j=1}^n (1- \alpha_{i, p} \alpha_{j, p} p^{-s})^{-1}, \hspace{3mm} L(\text{Sym}^2 \pi , s) \overset{.}=  \prod_p  \prod_{1 \leq i \leq j \leq n} (1- \alpha_{i, p} \alpha_{j, p} p^{-s})^{-1}
\end{equation*}
and 
\begin{equation}
L(\wedge^2 \pi, s) \overset{.}=  \prod_p  \prod_{1 \leq i< j \leq n} (1- \alpha_{i, p} \alpha_{j, p} p^{-s})^{-1}. 
\end{equation}
Here, $\overset{.}=$ means that equality holds up to multiplication by finitely many Euler factors at the ramified places, whose explicit description we omit; the equalities with $\overset{.}=$ are sufficient for our purposes since we will only be interested in the order of vanishing of the corresponding $L$-functions. From the Euler product expansions above, we see that 
\begin{equation} \label{fact}
L(\pi \otimes \pi, s) \overset{.}= L(\text{Sym}^2\pi, s) L( \wedge^2 \pi, s). 
\end{equation}
On the other hand, the second moment $L$-function $L_2(s, \pi)$ defined in Equation \eqref{sm} can be written as
$$
L_2(s, \pi ) \overset{.}= \frac{L(\text{Sym}^2 \pi , s)}{L(\wedge^2 \pi, s)}.  
$$

As explained in \cite[Example 1]{Devin2020}, there exists an open subset $U \supseteq \{s \in \mathbb C: \Re(s) \geq 1\}$ such that $L_2(s, \pi)$ can be continued to a meromorphic function on $U$; thus $\ord_{s=1} L_2(s, \pi)$ is well-defined.

\begin{theorem} [{Chebyshev's bias for Satake parameters \cite[Theorem 4.1]{Sheth24B}}] \label{satakebias}
Let $\pi$ be an irreducible cuspidal automorphic representation of $\GL_n(\mathbb A_\mathbb Q)$ such that $L(s, \pi)$ is entire, let $m= \ord_{s=1/2} L(s, \pi)$ and let $R(\pi)= \ord_{s=1} L_2(s, \pi)$. Assume Conjecture \ref{DRH} for $L(s, \pi)$. Then there exists a constant $c_\pi$ such that 
$$
\Re\left(\sum_{p \leqslant x} \frac{\alpha_{1, p}+ \cdots +\alpha_{n, p}}{\sqrt p} \right)= \left( \frac{R(\pi)}{2}-m \right) \log \log x+ c_\pi+ o(1). 
$$
\end{theorem}

\begin{proof}
The asymptotic in the theorem follows by taking logarithms in Conjecture \ref{DRH}, applying a generalised version of Mertens' theorem (\cite[Lemma 3.6]{Sheth24B}) and then taking real parts; we refer to the proof of \cite[Theorem 4.1]{Sheth24B} for more details. 
\end{proof}

\begin{proposition} \label{secondmoment}
Let $f \in S_k(\Gamma_0(N))$ be a cuspidal newform. Then we have that 
$
R(\pi_f) = 1, 
$
where $\pi_f$ denotes the automorphic representation generated by $f$. 
\end{proposition}

\begin{proof}
By \cite[Theorem 1.1]{Sha97} we have that $L(\text{Sym}^2 \pi_f, s)$  is non-vanishing on the line $\textrm{Re}(s)=1$. Moreover, $L( \pi_f \otimes  \pi_f, s)$ has a simple pole at $s=1$ since $\pi_f$ is self-dual (see for instance the proof of \cite[Theorem 3.4]{Devin2020}). On the other hand, we have by Equation \eqref{ap} that $L(\wedge^2 \pi_f, s) \overset{.}=\zeta(s)$. It follows from Equation \eqref{fact} that $ \ord_{s=1} L(\text{Sym}^2  \pi_f, s)=0$ and so
$\displaystyle{
R(\pi_f)= \ord_{s=1} L_2(s, \pi_f) = \ord_{s=1} L(\text{Sym}^2  \pi_f, s)-  \ord_{s=1} L(\wedge^2  \pi_f, s) = 0-(-1)=1.}$ 
\end{proof}

\begin{remark}
If $\pi$ is an arbitrary self-dual irreducible cuspidal automorphic representation, then  $R(\pi)$ can only equal $\pm 1$; cf. the proof of \cite[Theorem 3.4]{Devin2020}. 
\end{remark}

\begin{proof}[Proof of Theorem \ref{mainthm6}]
This follows by combining Equation \eqref{ap}, Theorem \ref{satakebias} and Proposition \ref{secondmoment}.
\end{proof}

\section{Explicit examples} \label{4}

Since Theorem \ref{mainthm6} shows that bias is completely determined by $m(f)$, in this section we explicitly compute these quantities for various examples of cuspidal Hecke eigenforms for the group $\SL_2(\mathbb Z)$.  

\subsection{Vanishing caused by root number considerations} 

The completed $L$-function of $f$ is defined by  
$
\Lambda(f, s) =   (2 \pi)^{-s-\frac{k-1}{2}}  \Gamma \left(s+ \frac{k-1}{2} \right)  L(s, f) 
$
and satisfies the functional equation 
\begin{equation} \label{fe}
\Lambda(f, s) = (-1)^{\frac{k}{2}} \Lambda(f, 1-s). 
\end{equation}

We note that if $k \equiv 2 \textrm{ mod } 4$, then it follows from Equation \eqref{fe} that $\Lambda \left(f, \frac{1}{2} \right)=0$; thus, $L\left(\frac{1}{2}, f\right)=0$ as well and $m(f) \geq 1$. We thus obtain the following corollary. 

\begin{corollary}
Assume Conjecture \ref{DRH}. If $f$ is a cuspidal Hecke eigenform for $\SL_2(\mathbb Z)$ of weight $k \equiv 2 \textrm{ mod } 4$, then the sequence $(a_f(p))_p$ has a Chebyshev's bias towards being negative. 
\end{corollary}

On the other hand, when $k \equiv 0 \textrm{ mod } 4$, it has been conjectured that $m(f)=0$; see, for instance, the work of Conrey--Farmer \cite{CF99} which introduces this conjecture and provides numerical evidence.  As another piece of evidence,  Luo \cite{Luo15} showed using the method of mollifiers that a positive proportion of Hecke eigenforms have non-vanishing central $L$-value as $k \to \infty$ with $k \equiv 0 \textrm{ mod } 4$. 
\\
It is natural to investigate $m(f)$ when $k \equiv 2 \textrm{ mod } 4$ as well. To the best of our knowledge, a precise conjecture for this case has not been recorded in the literature, although one expects that $m(f)=1$ by the minimalist philosophy that
$L$-functions do not have extra vanishing at the central point unless there is an underlying deeper geometric reason. Moreover,  Liu \cite{Liu18} showed that a positive proportion of Hecke eigenforms satisfy $L'(\frac{1}{2}, f) \neq 0$ as $k \to \infty$ with $k \equiv 2 \textrm{ mod } 4$; an explicit proportion was subsequently calculated by Jobrack \cite{Job20}. In view of the above, we make the following conjecture. 

\begin{conjecture}
Let $f$ be a cuspidal Hecke eigenform for the group $\SL_2(\mathbb Z)$ of weight $k$. If $k \equiv 2 \mod 4$,  then $m(f)=1$. 
\end{conjecture}

\begin{remark}
The above discussion yields the following interpretation of Theorem \ref{mainthm6}: the Fourier coefficients $a_f(p)$ have a bias towards being positive (resp. negative) when the root number of  $L(s, f)$ is positive (resp. negative). 
\end{remark}

\subsection{Super-positivity}

In their celebrated work on higher Gross--Zagier formulae over function fields, Yun--Zhang introduced the notion of super-positivity for self-dual automorphic $L$-functions (see \cite[Appendix B]{YZ17}) which asserts that the value of any derivative of the corresponding completed $L$-function at the central point must be non-negative. In the case of modular forms, this property is formulated as follows. 

\begin{definition}[Super-positivity] \label{sp}
Let $f$ be a cuspidal Hecke eigenform for the group $\SL_2(\mathbb Z)$ with associated completed $L$-function $\Lambda(s, f)$. We say that $\Lambda(s, f)$ satisfies the super-positivity property if $\Lambda^{(m)} \left( \frac{1}{2}, f \right) \geq 0$ for all $m \geq 0$. 
\end{definition}

 Yun--Zhang show that GRH implies the super-positivity property (\cite[Theorem B.2]{YZ17}). In the next subsection, we verify super-positivity \textit{unconditionally} for some explicit examples of modular forms (including Ramanujan's $\Delta$ function). Our argument also computes $m(f)$ in these examples, and we thus obtain a completely explicit version of Theorem \ref{mainthm6} for these examples.

\subsection{Mizumoto's positivity} \label{miz}
To study Definition \ref{sp}, we introduce another positivity property which was studied by Mizumoto in \cite{Miz99}. 

\begin{definition}[Mizumoto's positivity]
We say that a Hecke eigenform $f \in \textrm{M}_k(\SL_2(\mathbb Z))$ satisfies Mizumoto's positivity if $f(it) >0$ for all $t \in \mathbb R$ with $t>1$. 
\end{definition}

\begin{example} \label{miz2}
It follows from the product expansion
$\displaystyle{
\Delta(z)=q\prod_{n=1}^{\infty} (1-q^n)^{24}}$, that Ramanujan's $\Delta$-function satisfies Mizumoto's positivity. 
\end{example}

For $k \geq 1$, we let $$E_{2k}(z)= 1-\frac{4k}{B_{2k}} \sum_{n=1}^{\infty} \sigma_{2k-1}(n) q^n$$ be the Eisenstein series of weight $2k$ for the group $\SL_2(\mathbb Z)$. 

\begin{proposition} \label{miz1}
For all $k \geq 2$, $E_{2k}(z)$ satisfies Mizumoto's positivity. 
\end{proposition}

\begin{proof}
We recall that $E_{2k}$ satisfies the automorphy relation
$
E_{2k} \left(-\frac{1}{i} \right ) = (-1)^{k} E_{2k}(i). 
$
Thus if $k=2m+1$ for some $m \geq 1$, then $E_{2k}(i)=E_{4m+2}(i)=0$. Since $B_{4m+2}>0$, it follows that  
$
E_{2k}(it)>E_{2k}(i)=0 
$
for all $t>1$. If $k=2m$ for some $m > 1$, then since $B_{4m}<0$ we have 
$
E_{2k}(it)=E_{4m}(it)= 1-\frac{8m}{B_{4m}} \sum_{n=1}^{\infty} \sigma_{4m-1}(n) e^{- 2 \pi  n t}>0 
$ 
for all $t >1$. 
\end{proof}

For  $k \in \{12, 16, 20, 18, 22, 26\}$, we define 
$$
\Delta_k(z)= \Delta(z) \cdot E_{k-12}(z)=\sum_{n=1}^{\infty} \tau_k(n) q^n, 
$$
where we use the convention that $E_0=1$. Since $\dim_{\mathbb C}S_k(\SL_2(\mathbb Z))=1$ for $k \in \{12, 16, 20, 18, 22, 26\}$, it follows that $\Delta_k$ is a Hecke eigenform. By the main result of \cite{Duke1999} and \cite{Gha00, Gha02}, $\Delta_k$ for $k \in \{12, 16, 20, 18, 22, 26\}$ are the only eigenforms for $\SL_2(\mathbb Z)$ that are the product of two eigenforms. 

\begin{corollary} \label{deltak}
For  $k \in \{12, 16, 20, 18, 22, 26\}$, $\Delta_k$ satisfies Mizumoto's positivity. 
\end{corollary}

\begin{proof}
This follows by combining Example \ref{miz2} and Proposition \ref{miz1}. 
\end{proof}

\begin{theorem} \label{superpositivity}
The following hold true. 

\begin{enumerate}

\item For $k \in \{12, 16, 20, 18, 22, 26\}$, $\Lambda(\Delta_k, s)$ satisfies the super-positivity property.  

\item If $k=12, 16$ or $20$, then $m(\Delta_k)=0$. 

\item If $k=18, 22$ or $26$, then $m(\Delta_k)=1$.  

\end{enumerate}

\end{theorem}

\begin{proof}
Expanding the definition of $\Lambda(\Delta_k, s)$, making the substitution $y \mapsto \frac{1}{y}$ and applying the automorphy relation for $\Delta_k$ yields (cf. \cite[Theorem 9.7]{KKS12})
\begin{align*}
\Lambda(\Delta_k, s) &= \int_{0}^{\infty} \left(\sum_{n=1}^{\infty} \tau_k(n) e^{-2 \pi n y}  \right)y^{s+\frac{k-1}{2}-1} dy 
= \int_{0}^{\infty}  \Delta_k(iy) y^{s+\frac{k-1}{2}-1} dy \\
&=  \int_{0}^{1}  \Delta_k(iy) y^{s+\frac{k-1}{2}-1} dy +  \int_{1}^{\infty}  \Delta_k(iy) y^{s+\frac{k-1}{2}-1} dy  \\
&= \int_{1}^{\infty}  \Delta_k \left( \frac{i}{y} \right) y^{-s-\frac{k-1}{2}-1} dy +  \int_{1}^{\infty}  \Delta_k(iy) y^{s+\frac{k-1}{2}-1} dy  \\
&= (-1)^{\frac{k}{2}} \int_{1}^{\infty}  \Delta_k(iy) y^{\frac{k-1}{2}-s} dy +  \int_{1}^{\infty}  \Delta_k(iy) y^{s+\frac{k-1}{2}-1} dy   \\
&= \int_{1}^{\infty}  \Delta_k(iy) (  (-1)^{\frac{k}{2}} y^{\frac{k-1}{2}-s}+  y^{s+\frac{k-3}{2}} )dy. 
\end{align*}

We thus get
\begin{equation} \label{sup}
\Lambda^{(m)} \left (\Delta_k , \frac{1}{2} \right) = \left(1+(-1)^{\frac{k}{2}+m} \right)  \int_{1}^{\infty}  \Delta_k(iy) (\log y)^{m}    y^{\frac{k-2}{2}}dy 
\end{equation}

The assertions (1), (2) and (3) follow by combining Equation \eqref{sup} with Corollary \ref{deltak}.  \qedhere 

\end{proof}

\begin{corollary}
Assume Conjecture \ref{DRH} for $L(s, \Delta_k)$. We have that 
\begin{equation*}\label{tauk}
\sum_{p \leqslant x} \frac{\tau_{k}(p)}{p^{\frac{k}{2}}} = 
	\begin{cases}
	\dfrac{1}{2} \log \log x+c_k+o(1) & \text{if $k = 12, 16, 20 $} \\
	-\dfrac{1}{2} \log \log x+c_k+o(1) &\text{if $k = 18, 22, 26$}
\end{cases}
\end{equation*}
for some constant $c_k$.  In particular, in the sense of Definition \ref{Chebyshev}, the sequence $(\tau_k(p) p^{-\frac{k-1}{2}})_p$ has Chebyshev's bias towards being positive if $k \in \{12, 16, 20\}$ and a Chebyshev's bias towards being negative if $k \in \{18, 22, 26\}$.
\end{corollary}

\begin{proof}
This follows by combining Theorem \ref{mainthm6} and Theorem \ref{superpositivity}. \qedhere
\end{proof}

\appendix 

\section{Modular forms with real Fourier coefficients}

In Theorem \ref{mainthm6}, we restricted to modular forms with trivial nebentypus; Theorem \ref{rel} below shows that this restriction is necessary to study Chebyshev's bias in the framework of Definition \ref{Chebyshev}.

\begin{theorem} \label{rel}
Let $f \in S_k(\Gamma_0(N), \chi)$ be a newform with $a_f(n) \in \mathbb R$ for all $n \geq 1$.  Then either $\chi$ is trivial or $f$ is a CM form. 
\end{theorem}

\begin{proof}
By \cite[Equation (6.57)]{Iwa97}, we have that 
$
\overline{a_f(n)}= \chi(n) a_f(n) 
$ for all $n$ with $(n, N)=1$. Thus, if $a_f(n) \in \mathbb R$ for all $n \geq 1$, then either $\chi$ is trivial or $a_f(n)=0$ whenever $\chi(n) \neq 0$. In the latter case, we must have $a_f(p)=0$ for a set of primes having non-zero density. This implies, by a result of Serre \cite[Corollaire 2, p.174]{Ser81}, that $f$ is a CM form. 
\end{proof}

\end{document}